\documentclass{amsart}
\usepackage{amsmath}
\usepackage{graphicx}
\newtheorem{thm}{Theorem}[section]
\newtheorem*{thm*}{Theorem}

\newtheorem{lem}[thm]{Lemma}

\theoremstyle{definition}
\newtheorem{defn}[thm]{Definition}
\theoremstyle{remark}
\newtheorem{rem}[thm]{Remark}

\numberwithin{equation}{section}

\newcommand{\abs}[1]{\left\vert#1\right\vert}
\newcommand{\set}[1]{\left\{#1\right\}}
\newcommand{\paren}[1]{\left(#1\right)}

\newcommand{\inv}[1]{\frac{1}{#1}}

\newcommand{\ov}[1]{\overline#1}
\newcommand{\ha}[1]{\widehat{#1}}
\newcommand{\nha}[1]{\widetilde{#1}}
\newcommand{\ca}[1]{\mathcal{#1}}

\newcommand{\ppp}{\ldots}

\newcommand{\R}{\mathbb R}

\newcommand{\N}{\mathbb N}

\newcommand{\la}{\lambda}

\newcommand{\m}{\mathfrak{m}}
\newcommand{\2}{/}
\begin{document}

\title[A Marstrand's projection type theorem]{A generalization of Marstrand's theorem for projections of cartesian products}%
\author{Jorge Erick L\'{o}pez \and Carlos Gustavo Moreira}%
\address{IMPA - Estrada D. Castorina, 110 - 22460-320 - Rio de Janeiro - RJ - Brasil}%
\email{jelv@impa.br; gugu@impa.br}%


\begin{abstract}
We prove the following variant of Marstrand's theorem about projections of cartesian products of sets:

Let $K_1,\ldots,K_n$ Borel subsets of $\mathbb R^{m_1},\ldots
,\mathbb R^{m_n}$ respectively, and $\pi:\mathbb R^{m_1}\times\ldots\times\mathbb R^{m_n}\to\mathbb R^k$ be a surjective linear map. We set
$$\mathfrak{m}:=\min\left\{\sum_{i\in I}\dim_H(K_i) + \dim\pi(\bigoplus_{i\in I^c}\mathbb R^{m_i}), I\subset\{1,\ldots,n\}, I\ne\emptyset\right\}.$$
Consider the space $\Lambda_m=\{(t,O), t\in\mathbb R, O\in SO(m)\}$ with the natural measure and set $\Lambda=\Lambda_{m_1}\times\ldots\times\Lambda_{m_n}$. For every $\lambda=(t_1,O_1,\ldots,t_n,O_n)\in\Lambda$ and every $x=(x^1,\ldots,x^n)\in\mathbb R^{m_1}\times\ldots\times\mathbb R^{m_n}$ we define
$\pi_\lambda(x)=\pi(t_1O_1x^1,\ldots,t_nO_nx^n)$.
Then we have

\begin{thm*}

\emph{(i)} If $\mathfrak{m}>k$, then $\pi_\la(K_1\times\ppp\times K_n)$ has positive $k$-dimensional Lebesgue measure for almost every $\la\in\Lambda$.

\emph{(ii)} If $\mathfrak{m}\leq k$ and $\dim_H(K_1\times\ppp\times K_n)=\dim_H(K_1)+\ppp+\dim_H(K_n)$, then $\dim_H(\pi_\la(K_1\times\ppp\times K_n))=\mathfrak{m}$ for almost every $\la\in\Lambda$.
\end{thm*}
\end{abstract}
\maketitle
\section{Introduction}
The behavior of dimensions of projections of subsets of euclidean spaces has been studied for decades.

Let us denote by $\dim_H(X)$ the Hausdorff dimension of the set $X$. For $n$ and $k$ integers with $0<k<n$, $G(n,k)$ denotes the Grassmann manifold of all $k$-dimensional subspaces of $\R^n$, with the natural measure. For $V\in G(n,k)$, $P_V:\R^n\to V$ is the orthogonal projection onto $V$. The following is a fundamental result on dimensions of projections:
\begin{thm*}[Marstrand-Kaufman-Mattila]
Let $E\subset\R^n$ a Borel set. Then:

\emph{(i)} If $\dim_H(E)>k$, then $P_V(E)$ has positive $k$-dimensional Lebesgue measure for almost every $V\in Gr(n,k)$.

\emph{(ii)} If $\dim_H(E)\leq k$, then $\dim_H(P_V(E))=\dim_H(E)$ for almost every $V\in Gr(n,k)$.
\end{thm*}
This theorem was first proven by Marstrand \cite{Marstrand_1954} in 1954 for planar sets. Marstrand's proof used geometric methods. Later, Kaufman \cite{Kaufman_1968} gave an alternative proof of the same result using potential-theoretic methods. Finally, Mattila \cite{Mattila_1975} generalized it to higher dimensions; his proof combined the methods of Marstrand and Kaufman.

There are other variants of the Marstrand-Mattila's theorem. They were unified in a more general result due to Peres and Schlag \cite{Peres_Schlag_2000}. They studied general smooth families of projections, using some methods from harmonic analysis. The crucial characteristic that is common to all families of projections considered in Peres-Schlag's result is a transversality property (see \cite{Peres_Schlag_2000}, Definition 7.2).

We are interested in a Marstrand's projection result that actually is outside of the Peres-Schlag's scheme (the families of projections considered here, in general, are not transversal). This result was motivated by the problem of understanding the behavior of projections of cartesian products of sets product of sets, by a fixed projection map.

Let $K_1,\ppp,K_n$ Borel subsets of $\R^{m_1},\ppp,\R^{m_n}$ respectively, and  $\pi:\R^{m_1}\times\ppp\times\R^{m_n}\to\R^k$ be a linear map. Then
\begin{equation}\label{2^_conditions}
\dim_H(\pi(K_1\times\ppp\times K_n))\leq\min\set{\sum_{i\in I}\dim_H(K_i) + \dim\pi(\bigoplus_{i\in I^c}\R^{m_i}), I\subset\set{1,\ppp,n}},
\end{equation}
with the conventions $\sum_{i\in\emptyset}\dim_H(K_i)=0, \dim\emptyset =0$.

Consider the space $\Lambda_m=\set{(t,O), t\in\R, O\in SO(m)}$ with the natural measure and set $\Lambda=\Lambda_{m_1}\times\ppp\times\Lambda_{m_n}$. For every $x=(x^1,\ppp,x^n)\in\R^{m_1}\times\ppp\times\R^{m_n}$ and every  $\la=(t_1,O_1,\ppp,t_n,O_n)\in\Lambda$ we define $\pi_\la(x)=\pi(t_1O_1x^1,\ppp,t_nO_nx^n)$. Suppose that $\pi$ is surjective and set $$\mathfrak{m}:=\min\set{\sum_{i\in I}\dim_H(K_i) + \dim\pi(\bigoplus_{i\in I^c}\R^{m_i}), I\subset\set{1,\ppp,n}, I\ne\emptyset}.$$
Then we have

\begin{thm}\label{Marstrand_theorem}

\emph{(i)} If $\mathfrak{m}>k$, then $\pi_\la(K_1\times\ppp\times K_n)$ has positive $k$-dimensional Lebesgue measure for almost every $\la\in\Lambda$.

\emph{(ii)} If $\mathfrak{m}\leq k$ and $\dim_H(K_1\times\ppp\times K_n)=\dim_H(K_1)+\ppp+\dim_H(K_n)$, then $\dim_H(\pi_\la(K_1\times\ppp\times K_n))=\mathfrak{m}$ for almost every $\la\in\Lambda$.
\end{thm}

In a work in progress, we plan use the Theorem \ref{Key_Theorem} to generalize the result of Moreira and Yoccoz \cite{Moreira_Yoccoz_2001} about stable intersections of two regular Cantor sets for projections of cartesian products of several regular Cantor sets. Our goal is to prove the following result: for any given surjective linear map $\pi:\R^{n}\to\R^k$, typically for regular Cantor sets on the real line $K_1,\ppp,K_n$ with $\mathfrak{m}>k$, the set $\pi(K_1\times\ppp\times K_n)$ persistently contains non-empty open sets of $\R^k$. Such a result would in particular imply an analogous result for simultaneous stable intersections of several regular Cantor sets on the real line.

In another work in progress, in collaboration with Pablo Shmerkin, we plan to use the results of this paper combined with the techniques in \cite{Hochman_Shmerkin_2009} in order to obtain exact formulas for the Hausdorff dimensions of projections of cartesian products of (real or complex) regular Cantor sets under explicit irrationality conditions.
\vskip .1 in
{\bf Acknowledgement:} We are grateful to P. Shmerkin for the useful discussions about the subject of this work.
\section{Statement the main results}
Let $\mu$ be a finite Borel measure on $\R^m$. The \emph{$s$-energy} of $\mu$ is $$I_s(\mu)=\int\int\frac{d\mu(x)d\mu(y)}{\abs{x-y}^s}.$$
We know (see \cite{Mattila_1995}, Theorem 8.9(3)) that for a Borel set $K\subset\R^m$
\begin{multline}\label{Frostman_Lemma}
\dim_H(K)=\sup\{s\in\R, \textrm{there is a compactly supported measure } \mu
\textrm{ on } K \\\textrm{ which } 0<\mu(\R^m)<\infty \textrm{ and } I_s(\mu)<\infty\}.
\end{multline}
The \emph{Fourier transform} of $\mu$ is denoted by $\ha{\mu}$ and defined as
 $$\ha{\mu}(\xi)=\int_{\R^m}e^{-i\xi\cdot x}d\mu(x).$$
It is well-know that if $\ha{\mu}\in L^2(\R^m)$, then $\mu$ is absolutely continuous with $L^2$-density.
Energy and Fourier transform are related as follow (see \cite{Mattila_1995}, Lemma 12.12)
$$I_s(\mu)=(2\pi)^{-m}c(s,m)\int\abs{\xi}^{s-m}\abs{\ha{\mu}(\xi)}^2d\xi,$$
for $0<s<m$ and $\mu$ with compact support.

We summarize the above observations as the following result:\\
Let  $F\subset\R^k$ a Borel set supporting a probability measure $\nu$ with $\int\abs{\xi}^{s-k}\abs{\ha{\nu}(\xi)}^2d\xi<\infty$. If $s\geq k$, then $F$ has positive $k$-dimensional Lebesgue measure. Otherwise, if $0<s<k$, then $\dim_H(F)\geq s$.

Let $\pi:\R^{m_1}\times\ppp\times\R^{m_n}\to\R^k$ be a linear map. For each $I\subset\set{1,\ppp,n}$, let $P_{I}:\R^{m_1}\times\ppp\times\R^{m_n}\to\R^{m_1}\times\ppp\times\R^{m_n}$ the orthogonal projection onto the subspace $\bigoplus_{i\in I}\R^{m_i}$, where $\R^{m_i}$ is as a canonical subspace of $\R^{m_1}\times\ppp\times\R^{m_n}$. Then $\pi=\pi\circ P_I+\pi\circ P_{I^c}$ so, for $K_1,\ppp,K_n$ Borel subsets of $\R^{m_1},\ppp,\R^{m_n}$ respectively we have
\begin{align*}
&\dim_H(\pi(K_1\times\ppp\times K_n))\\
&\leq \dim_H\Big(\pi P_I(K_1\times\ppp\times K_n)\times\pi P_{I^c}(K_1\times\ppp\times K_n)\Big)\\
&\leq \dim_H\Big(\pi P_I(K_1\times\ppp\times K_n)\times\pi(\bigoplus_{i\in I^c}\R^{m_i})\Big)\\
&\leq\sum_{i\in I}\dim_H(K_i)+\dim\pi(\bigoplus_{i\in I^c}\R^{m_i}).
\end{align*}
(In the last inequality, we assume that $\dim_H(K_1\times\ppp\times K_n)=\dim_H(K_1)+\ppp+\dim_H(K_n)$) This prove the inequality (\ref{2^_conditions}) and also motivates us to define:

\begin{defn}\label{definition_m}
For $\pi:\R^{m_1}\times\ppp\times\R^{m_n}\to\R^k$ a surjective linear map and $d_1,\ppp,d_n$ nonnegative real numbers, we define $\mathfrak{m}=\mathfrak{m}(\pi,d_1,\ppp, d_n)$ as
$$\mathfrak{m}=\min\set{\sum_{i\in I}d_i + \dim\pi(\bigoplus_{i\in I^c}\R^{m_i}), I\subset\set{1,\ppp,n}, I\ne\emptyset}.$$
\end{defn}

\begin{rem}
If in addition $d_i\leq m_i$ (which holds for dimensions of subsets of $\R^{m_i}$), then, for the open and total measure family of linear maps $\pi$ with the following transversality property:
$$\dim\pi(\bigoplus_{i\in I}\R^{m_i})=\min\big(k,\dim(\bigoplus_{i\in I}\R^{m_i})\big),\textrm{ for all }I\subset\set{1,\ppp,n},$$
the equivalence $\mathfrak{m}(\pi,d_1,\ppp,d_n)>k \Leftrightarrow d_1+\ppp+d_n>k$ holds. However, in general we must check more than one of the $2^n-1$ conditions appearing in the definition of $\mathfrak{m}$.
\end{rem}

Consider the space $\Lambda_m=\set{(t,O), t\in\R, O\in SO(m)}$, with the product measure $\mathcal{L}^1\times\Theta^m$, where $\ca{L}^1$ denotes the one dimensional Lebesgue measure and $\Theta^m$ denotes the left-right invariant Haar probability measure on $SO(m)$. Notice that the set $C(m)=\set{tO, t\in\R, O\in SO(m)}$ represents essentially the family of linear conformal maps on $\R^m$. $C(2)=\set{\left( \begin{array}{cc} a &-b \\ b & a \\ \end{array} \right), a,b\in\R}$, which can be viewed as the set of multiplications by a complex number.

We set $\Lambda=\Lambda_{m_1}\times\ppp\times\Lambda_{m_n}$. For every $x=(x^1,\ppp,x^n)\in\R^{m_1}\times\ppp\times\R^{m_n}$, and every  $\la=(t_1,O_1,\ppp,t_n,O_n)\in\Lambda$ we define $\pi_\la(x)=\pi(t_1O_1x^1,\ppp,t_nO_nx^n)$.
Also, given any finite measure $\mu$ on $\R^{m_1}\times\ppp\times\R^{m_n}$, let $\nu_\la=(\pi_\la)_*\mu$. We also  define
$$I_{d_1,\ppp,d_n}(\mu)=\int\int\frac{d\mu(x)d\mu(y)}{\abs{x^1-y^1}^{d_1}\ppp\abs{x^n-y^n}^{d_n}}.$$
Our main result is now the following:

\begin{thm}\label{Key_Theorem} Let $\pi$ and $d_1,\ppp,d_n$ be as in definition \ref{definition_m} with $\m=\m(\pi, d_1,\ppp,d_n)\neq 0,1,\ppp,k-1$. Then, there exist $d'_1\leq d_1,\ppp,d'_n\leq d_n$ such that for every Borel measure $\mu$ on $\R^{m_1}\times\ppp\times\R^{m_n}$ we have
$$\int_{\Lambda}\int_{\R^k}\abs{\xi}^{\mathfrak{m}-k}\abs{\ha{\nu_\la}(\xi)}^2\rho(\la)d\xi d\la\leq C_{\mathfrak{m}}I_{d'_1,\ppp,d'_n}(\mu),$$
where $\rho(\la)=|t_1|^{m_1-1}\ppp|t_n|^{m_n-1}e^{-\inv{2}(|t_1|^2+\ppp+|t_n|^2)}$ and $C_{\mathfrak{m}}>0$ is some constant depending only on $\pi, n, k, m_1,\ppp,m_n$ and $\mathfrak{m}$.
\end{thm}

In the proof of Theorem \ref{Key_Theorem} the key tool will be the following combinatorial lemma.

\begin{lem}[Weights Lemma]\label{Existence_of_weights}
Let $s,d_1,\ppp,d_n\geq0$ and $V_1,\ppp,V_n$ vector subspaces of a same finite dimension vector space satisfying the following $2^n$ conditions
$$\sum_{i\in I}d_i + \dim \big(\sum_{i\in I^c}V_i\big)\geq  s, \textrm{ for every }I\subset\set{1,\ppp,n}$$
(with the conventions $\sum_{i\in\emptyset}d_i=0$, $\dim\emptyset=0$).

Fixed a generating set $\set{v^i_{1},\ppp,v^i_{m_i}}$ of $V_i$ for each $i\in\set{1,\ppp,n}$. Consider the family $\mathbb{J}$ of all possible  $J=(\textsc{j}_1,\ppp,\textsc{j}_n)$, $\textsc{j}_i\subset\set{v^i_{1},\ppp,v^i_{m_i}}$ such that $\textsc{j}_1\cup\ppp\cup\textsc{j}_n$ is a linearly independent system with dimension greater than or equal to $s$. Define
$$\mathbb{\ov{J}}=\set{(J,i)\in\mathbb{J}\times\set{1,\ppp,n}, \ha{J}(i):=(\#\textsc{j}_1,\ppp,\#\textsc{j}_n)+(s-(\#\textsc{j}_1+\ppp+\#\textsc{j}_n))e_i\geq0},$$
where $e_1,\ppp,e_n$ is the canonical basis of $\R^n$ and $\geq$ means that the inequality is coordinate to coordinate.

Then, there exist non-negative real numbers $(\alpha_{(J,i)})_{(J,i)\in\mathbb{\ov{J}}}$ with sum equal to $1$ such that
$$\sum_{(J,i)\in\mathbb{J}}\alpha_{(J,i)}\ha{J}(i)\leq d.$$
\end{lem}

\begin{proof}[Proof of Theorem \ref{Marstrand_theorem}]
The theorem follows immediately from the Theorem \ref{Key_Theorem} applied to $\mu=\mu_1\times\ppp\times\mu_n$ for suitable measures $\mu_i$ compactly supported in $K_i$ coming from the equation (\ref{Frostman_Lemma}). Noting that in the part (i), the condition $\dim_H(K_1)>0,\ppp,\dim_H(K_n)>0$ follows from the hypotheses; and in the part (ii), we may assume the same condition by reduction to some cartesian product if necessary.
\end{proof}


\begin{rem}
We can derive the part (ii) of the Theorem \ref{Marstrand_theorem} from the part (i). Assume $\dim_H(K_i)>0$. Let $k'<\m\leq k'+1\leq k$ and consider any $k'<s<\m$, and set $\Lambda^s=\set{\la\in\Lambda,\dim_H(\pi_\la(K_1\times\ppp\times K_n))<s}$. The idea is to add another factor to the cartesian product: Let $m_0:=k-k'$ and consider $K_0$ a sufficiently regular subset of $\R^{m_0}$ with $\dim_H(K_0)=k-s$, and $\widetilde{\pi}:\R^{m_0}\times\R^{m_1}\times\ppp\times\R^{m_n}\to \R^k$ with
\begin{align*}
\widetilde{\pi}\circ P_{I^n}=&\pi, \textrm{ where } I^n=\set{1,\ppp,n},\\
\dim\widetilde{\pi}\big(\bigoplus_{i\in I\cup\set{0}}\R^{m_i})\big)=&\min\big(k,m_0+\dim\pi(\bigoplus_{i\in I}\R^{m_i})\big), \textrm{ for all } I\subset\set{1,\ppp,n}.
\end{align*}
 In particular
$\widetilde{\pi}$ is surjective. Notice that
$$\sum_{i\in I}\dim_H(K_i) + \dim\widetilde{\pi}(\bigoplus_{i\in I^c}\R^{m_i})>k, \textrm{ for all }I\subset\set{0,1,\ppp,n}, I\neq\emptyset,$$
and also that $\dim_H(\nha{\pi}_{(\la_0,\la)}(K_1\times\ppp\times K_n))<k$ for all $(\la_0,\la)\in\Lambda_{m_0}\times\Lambda^s$. Applying the Theorem \ref{Marstrand_theorem}.(i) in this new setting, we conclude that $\Lambda^s$ is a zero measure subset of $\Lambda$.
\end{rem}

\begin{rem}
Theorem \ref{Key_Theorem}, when combined with Proposition 7.5 of \cite{Peres_Schlag_2000}, also gives us a result on exceptional sets:

In the setting of the Theorem \ref{Marstrand_theorem}, part  $(i)$, we have
$$\dim_H\big(\set{\la\in\Lambda, t_i\neq0 \textrm{\emph{ if }} m_i>1, \mathcal{L}^k(\pi_\la(K_1\times\ppp\times K_n))=0}\big)\leq l+k-\mathfrak{m},$$
where $l=\dim\Lambda_{m_1}\times\ppp\times\Lambda_{m_n}=n+\sum_{i=1}^{n}m_i(m_i-1)\22.$
\end{rem}

\section{Proof of the main results}\label{section_proof}

\begin{proof} [Proof of Theorem \ref{Key_Theorem}] Notice that
\begin{align*}
\abs{\ha{\nu_\la}(\xi)}^2&=\int\int e^{i\xi\cdot\pi_\la(y-x)}d\mu(x)d\mu(y),\\
&=\int\int e^{i\pi^{T}\xi\cdot(t_1O_1(y^1-x^1),\ppp,t_nO_n(y^n-x^n))}d\mu(x)d\mu(y),
\end{align*}
and that, for all $z\in\R^m$, $\eta\in\R^m$,
\begin{align*}
\int_\R\int_{SO(m)}e^{i\eta\cdot tOz}|t|^{m-1}e^{-\inv{2}\abs{t}^2}d\Theta^m dt&=
\int_\R\int_{S^{m-1}}e^{i|z|\eta\cdot t\theta}|t|^{m-1}e^{-\inv{2}\abs{t}^2}d\sigma^{m-1} dt\\
&=2\int_{\R^m}e^{i|z|\eta\cdot x}e^{-\inv{2}\abs{x}^2}dx\\
&=2\pi^{\frac{m}{2}} e^{-\inv{2}(|z||\eta|)^2},
\end{align*}
where $\sigma^{m-1}$ denotes the normalized Lebesgue measure on $S^{m-1}$. Therefore by Fubini's theorem
\begin{equation*}
\begin{split} &\int_\Lambda\int_{\R^k}\abs{\xi}^{\mathfrak{m}-k}\abs{\ha{\nu_\la}(\xi)}^2\rho(\la)d\xi d\la\\
&=\lim_{a\to\infty}\int_{\abs{\xi}\leq a}\int_\Lambda\abs{\xi}^{\mathfrak{m}-k}\abs{\ha{\nu_\la}(\xi)}^2\rho(\la)d\la d\xi\\
&=c\lim_{a\to\infty}\int\int\paren{\int_{\abs{\xi}\leq a}\abs{\xi}^{\mathfrak{m}-k}e^{-\inv{2}\abs{D_{x,y}(\xi)}^2}d\xi}d\mu(x)d\mu(y)\\ &= c\int\int\paren{\int_{\R^k}\abs{\xi}^{\mathfrak{m}-k}e^{-\inv{2}\abs{D_{x,y}(\xi)}^2}d\xi}d\mu(x)d\mu(y),\\ \end{split}
\end{equation*}
where $D_{x,y}=\paren{D^1(\abs{y^1-x^1}),\ppp,D^n(\abs{y^n-x^n})}\circ\pi^T$, and $D^i(t):\R^{m_i}\to\R^{m_i}$ is the diagonal transformation, $D^i(t)=t.Id$, for $t\in\R$.

We fixed $x,y$ assuming that $y^i-x^i\neq 0$ for all $i=1,\ppp,n$. We estimate $\int_{\R^k}\abs{\xi}^{\mathfrak{m}-k}e^{-\inv{2}\abs{D_{x,y}(\xi)}^2}d\xi$ separately, when $\m\geq k$ and $\m<k$. In both case we apply the Lemma \ref{Existence_of_weights} for $V_i=\pi(\R^{m_i})$, taking $v_j^i=\pi(e^i_j)$, where $e^{i}_j$, $j=1,\ppp,m_i$ is the canonical basis of $\R^{m_i}$ as subspace of $\R^{m_1}\times\ppp\times\R^{m_n}$.

We use the notation $z^I=(z_1^{i_1},\ppp,z_n^{i_n})$ if $z=(z_1,\ppp,z_n)\in\R_+^n$ and $I=(i_1,\ppp,i_n)\in\N^n$, for $z=(\abs{y^1-x^1},\ppp,\abs{y^n-x^n})$.
\vskip .1 in
Suppose $\m\geq k$. Let $i_0$ such that $z_{i_0}\leq z_{i}$ for all $i=1,\ppp,n$. Notice that $\m(\pi,d-(\m-k)e_{i_0})\geq k$ and in particular $d-(\m-k)e_{i_0}\geq0$. We apply the Lemma \ref{Existence_of_weights} to $d-(\m-k)e_{i_0}$ and $s=k$. For each $J\in\mathbb{J}$, just looking for the sums in $\inv{2}\abs{D_{x,y}(\xi)}^2$ related to $J$ and using the change of variables formula to an appropriated linear isomorphs of $\R^k$, we have $$\int_{\R^k}\abs{\xi}^{\mathfrak{m}-k}e^{-\inv{2}\abs{D_{x,y}(\xi)}^2}d\xi
\leq c'z_{i_0}^{k-\m}z^{-\ha{J}}\int_{\R^k}|\eta|^{\m-k}e^{-\inv{2}|\eta|^2}d\eta,$$
for some constant $c'>0$ depending only on $\pi$ and $\m-k$, where $\ha{J}:=(\#\textsc{j}_1,\ppp,\#\textsc{j}_n)$. Therefore
$$\int_{\R^k}\abs{\xi}^{\mathfrak{m}-k}e^{-\inv{2}\abs{D_{x,y}(\xi)}^2}d\xi
\leq c''z_{i_0}^{k-\m}\prod_{J\in\mathbb{J}}z^{-\alpha_J\ha{J}}=c''z^{-(\sum_J\alpha_J\ha{J}+(\m-k)e_{i_0})}=c''z^{-d'}.$$

Suppose $k'-1<\m<k'$, where $1\leq k'\leq k$. We apply the Lemma \ref{Existence_of_weights} to $d$ and $s=\m$. Let $(J,i)\in\mathbb{\ov{J}}$ with $\#\textsc{j}_1+\ppp+\#\textsc{j}_n=k'$. From $\m<k'$ we have $\textsc{j}_i\neq\emptyset$. In the same way as in the previous case, notice that
$$\int_{\R^k}\abs{\xi}^{\mathfrak{m}-k}e^{-\inv{2}\abs{D_{x,y}(\xi)}^2}d\xi
\leq \nha{c}z^{-\ha{J}}\int_{\R^{k'}}\int_{\R^{k-k'}}\paren{|\eta'_1|\2z_i+|\eta''|}^{\m-k}e^{-\inv{2}|\eta'|^2}d\eta'd\eta'',$$
for some constant $\nha{c}>0$ depending only on $\pi$ and $\m-k$. We affirm that
$$\int_{\R^{k'}}\int_{\R^{k-k'}}\paren{|\eta'_1|\2z_i+|\eta''|}^{\m-k}e^{-\inv{2}|\eta'|^2}d\eta'd\eta''\leq \nha{c}'z_i^{k'-\m},$$
for some constant $\nha{c}'>0$ depending only on $\m, k, k'$.
If $k'=k$ the affirmation is true, since $\m-k>-1$. If $k'<k$, applying polar coordinates in $\R^{k-k'}$ we have
\begin{align*}
\int_{\R^{k'}}\int_{\R^{k-k'}}\paren{|\eta'_1|\2z_i+|\eta''|}^{\m-k}e^{-\inv{2}|\eta'|^2}d\eta'd\eta''
\leq& C\int_{\R_+}\int_{\R_+}(t\2z_i+r)^{\m-k'-1}e^{-\inv{2}t^2}drdt\\
=&C(k'-m)^{-1}\int_{\R_+}(t\2z_i)^{\m-k'}e^{-\inv{2}t^2}dt.
\end{align*}
Then $\int_{\R^k}\abs{\xi}^{\mathfrak{m}-k}e^{-\inv{2}\abs{D_{x,y}(\xi)}^2}d\xi
\leq\nha{c}''z^{-\ha{J}(i)}$, and therefore
$$\int_{\R^k}\abs{\xi}^{\mathfrak{m}-k}e^{-\inv{2}\abs{D_{x,y}(\xi)}^2}d\xi
\leq\nha{c}''\prod_{(J,i)\in\mathbb{\ov{J}}}z^{-\alpha_{(J,i)}\ha{J}(i)}=\nha{c}''z^{-\sum_{(J,i)\in\mathbb{\ov{J}}}\alpha_{(J,i)}\ha{J}(i)}=\nha{c}''z^{-d'}.$$
\end{proof}

\begin{proof}[Proof of Lemma \ref{Existence_of_weights}] \emph{Claim:} The vertices of the polyhedron
\begin{align*}
P=\Big\{(d_1,\ppp,d_n)\in\R^n, d_1\geq0,\ppp,d_n&\geq0\\
\sum_{i\in I}d_i + \dim \big(\sum_{i\in I^c}V_i\big)&\geq s, \textrm{ for all } I\subset\set{1,\ppp,n}\Big\}
\end{align*}
have all the form $\ha{J}(i)$ for some $(J,i)\in\mathbb{J}$.
\vskip .1 in
$P\subset\overline{\R}^n_+$, therefore $P$ is a pointed polyhedron (i.e. it does not contain any non trivial affine subspace). We proceed by induction on $n$. For $n=1$ it is trivial. Let $x=(x_1,\ppp,x_n)$ any vertex of the polyhedron. Then, there are $n$ independent inequalities from the definition of $P$ that become equality in $x$ (see \cite{Schrijver_1986}, page 104).

If $x_n=0$, notice that $x'=(x_1,\ppp,x_{n-1})$ is now a vertex of the polyhedron
\begin{align*}
P'=\Big\{(d_1,\ppp,d_{n-1})\in\R^{n-1}, d_1\geq0,\ppp,d_{n-1}&\geq0\\
\sum_{i\in I}d_i + \dim \big(\sum_{i\in I^c}V_i\big)&\geq s, \textrm{ for all } I\subset\set{1,\ppp,n-1}\Big\}
\end{align*}
(i.e. $x'\in P'$ and $x'$ satisfies $n-1$ independent equalities). By induction hypothesis, there exist some   $J'=(\textsc{j}'_1,\ppp,\textsc{j}'_{n-1})\in\mathbb{J}'$ and $i'\in\set{1,\ppp,n-1}$  such that $x'=\ha{J'}(i')$. Then, $J=(\textsc{j}'_1,\ppp,\textsc{j}'_{n-1},\emptyset)\in\mathbb{J}$ and $i=i'$ are such that $x=\ha{J}(i)$.

Suppose $x_1\neq0,\ppp,x_n\neq0$. By simplicity, we denote $\sum_{i\in I}V_i$ by $V_I$. Consider
$$\mathcal{I}=\set{I\subset\set{1,\ppp,n}, I\neq\emptyset, \sum_{i\in I}x_i + \dim V_{I^c} = s}.$$
By the assumption on $x$, there are $I_1,\ppp,I_n\in\mathcal{I}$ such that the associated $0,1$ row vectors $\nha{I}_1,\ppp,\nha{I}_n$ defining the equalities, are independent.

If $I,J\in\mathcal{I}$, then
\begin{align*}
\dim V_{I^c}+\dim V_{J^c} &= 2s-\sum_{i\in I}x_i-\sum_{i\in J}x_i\\
&=2s-\sum_{i\in I\cup J}x_i-\sum_{i\in I\cap J}x_i\\
&\leq \dim V_{I^c\cap J^c} + \dim V_{I^c\cup J^c}\\
&\leq \dim (V_{I^c}\cap V_{J^c}) + \dim (V_{I^c}+ V_{J^c})\\
&= \dim V_{I^c}+\dim V_{J^c},
\end{align*}
therefore, $I\cup J\in \mathcal{I}$ and $I\cap J\in\mathcal{I}$. Let $I_0\in\mathcal{I}$ a minimal element by inclusion. Then, for any $J\in\mathcal{I}$, we have
$$I_0\subset J \textrm{ or } I_0\cap J=\emptyset.$$
This means the invertible matrix of rows $\nha{I}_1,\ppp,\nha{I}_n$ has $\#I_0$ identical columns, and therefore $\#I_0=1$, say $I_0=\set{n}$, or, equivalently, $x_n=s-\dim( V_1+\ppp+V_{n-1})$.

Notice that now $\nha{x}=(x_1,\ppp,x_{n-1})$ is a vertex of the polyhedron
\begin{align*}
\nha{P}=\Big\{(d_1,\ppp,d_{n-1})\in\R^{n-1}, d_1&\geq0,\ppp,d_{n-1}\geq0\\
\sum_{i\in I}d_i + \dim \big(\sum_{i\in I^c}V_i\big)&\geq \dim(V_1+\ppp+V_{n-1}), \textrm{ for all } I\subset\set{1,\ppp,n-1}\Big\}
\end{align*}
By induction hypothesis, there exist some appropriate  $\nha{J}=(\nha{\textsc{j}}_1,\ppp,\nha{\textsc{j}}_{n-1})\in\mathbb{\nha{J}}$ such that $\nha{x}=(\#\nha{\textsc{j}}_1,\ppp,\#\nha{\textsc{j}}_{n-1})$. We can take $\textsc{j}_n\subset\set{v^n_1,\ppp,v^n_{m_n}}$ such that $V_1+\ppp+V_{n-1}+\langle\textsc{j}_n\rangle=V_1+\ppp+V_{n}$ and $J=(\nha{\textsc{j}}_1,\ppp,\nha{\textsc{j}}_{n-1},\textsc{j}_n)\in\mathbb{J}$. Notice that $x=\ha{J}(n)$. This finishes the proof of the claim.
\vskip .1 in
To finish the prove of the lemma, notice that for a pointed polyhedron $P$ (see \cite{Schrijver_1986}, page 108), we have
$$P=\textrm{conv.hull}\set{x^1,\ppp,x^r}+\textrm{cone}\set{y^1,\ppp,y^t}$$
where $x^i$ are the vertices of $P$ and $y^i$ are its extremal rays; and we have necessary $y^i\geq 0$ since $P\subset\overline{\R}^n_+$.
\end{proof}

\begin{rem}
Notice that $\ha{J}(i)\in P$ for all $(J,i)\in\mathbb{J}$, hence we conclude from Lemma \ref{Existence_of_weights} that
$$P=\textrm{conv.hull}\{\ha{J}(i), (J,i)\in\mathbb{J} \}+\textrm{cone}\set{e_1,\ppp,e_n}.$$
\end{rem}

\bibliographystyle{amsplain}
\bibliography{bibi}
\end{document}